\documentclass[reqno,openright,dvips,11pt,myheadings,twoside]{amsart}

\usepackage{ifthen}

\newcounter{dtlForSubmission} 

 \newcounter{DetailLevel} 

\newcommand{\DetailMarginNote}[1]{
    \ifthenelse{\value{DetailLevel}=\value{dtlMarginComments} \or \value{DetailLevel}>\value{dtlMarginComments}}
        {{\small #1}}{}
    }

\newcommand{\DetailSome}[1]{
    \ifthenelse{\value{DetailLevel}=\value{dtlSomeDetail} \or \value{DetailLevel}>\value{dtlSomeDetail}}
        {{\small \textbf{Detailed compile only}: #1}}{}
    }

\newcommand{\DetailFull}[1]{
    \ifthenelse{\value{DetailLevel}=\value{dtlFullDetails} \or \value{DetailLevel}>\value{dtlFullDetails}}
        {{\small \textbf{Detailed compile only}: #1}}{}
    }

\newcommand{\NotDetailSome}[1]{
    \ifthenelse{\value{DetailLevel}=\value{dtlSomeDetail} \or \value{DetailLevel}>\value{dtlSomeDetail}}
        {}{#1}
    }

\newcommand{\NotDetailFull}[1]{
    \ifthenelse{\value{DetailLevel}=\value{dtlFullDetails} \or \value{DetailLevel}>\value{dtlFullDetails}}
        {}{#1}
    }

\newcommand{\DetailSomeElse}[2]{
    \ifthenelse{\value{DetailLevel}=\value{dtlSomeDetail} \or \value{DetailLevel}>\value{dtlSomeDetail}}
        {{\small \textbf{Detailed compile only}: #1}}{#2}
    }

\newcommand{\DetailFullElse}[2]{
    \ifthenelse{\value{DetailLevel}=\value{dtlFullDetails} \or \value{DetailLevel}>\value{dtlFullDetails}}
        {{\small \textbf{Detailed compile only}: #1}}{#2}
    }

\newcommand{\DetailSomeInline}[1]{
    \ifthenelse{\value{DetailLevel}=\value{dtlSomeDetail} \or \value{DetailLevel}>\value{dtlSomeDetail}}
        {{\small #1}}{}
    }

\newcommand{\DetailFullInline}[1]{
    \ifthenelse{\value{DetailLevel}=\value{dtlFullDetails} \or \value{DetailLevel}>\value{dtlFullDetails}}
        {{\small #1}}{}
    }

\newcommand{\DetailSomeElseInline}[2]{
    \ifthenelse{\value{DetailLevel}=\value{dtlSomeDetail} \or \value{DetailLevel}>\value{dtlSomeDetail}}
        {{\small #1}}{#2}
    }

\newcommand{\DetailFullElseInline}[2]{
    \ifthenelse{\value{DetailLevel}=\value{dtlFullDetails} \or \value{DetailLevel}>\value{dtlFullDetails}}
        {{\small #1}}{#2}
    }

\newcommand{\ExplainDetailLevel}{
    Detail level is
    \ifthenelse{\value{DetailLevel}=\value{dtlForSubmission}}
        {0: for submission}
        {\ifthenelse{\value{DetailLevel}=\value{dtlMarginComments}}
            {1: as for submission but with margin comments}
            {\ifthenelse{\value{DetailLevel}=\value{dtlSomeDetail}}
                {2: some proofs not intended for submission}
               {\ifthenelse{\value{DetailLevel}=\value{dtlFullDetails}}
                   {3: full details}
                   {invalid}
                }
            }
        }
    }

\newcommand{\Ignore}[1]{}

\newtheorem{theorem}{Theorem}[section]
\newtheorem{lemma}[theorem]{Lemma}
\newtheorem{prop}[theorem]{Proposition}
\newtheorem{cor}[theorem]{Corollary}

\theoremstyle{definition}
\newtheorem{definition}[theorem]{Definition}

\newtheorem{remark}[theorem]{Remark}

\numberwithin{equation}{section}

\newcommand{\abs}[1]{\left\vert#1\right\vert}

\newcommand{\skipline}{\vspace{11pt}}

\newcommand{\real}{\mathbb{R}}
\newcommand{\BB}[1]{\ensuremath{\mathbb{#1}}}
\newcommand{\R}{\ensuremath{\BB{R}}} %
\newcommand{\iny}{\ensuremath{\infty}}
\newcommand{\grad}{\ensuremath{\nabla}}
\DeclareMathOperator{\dv}{div} %
\DeclareMathOperator{\curl}{curl} %
\DeclareMathOperator{\dist}{dist} %
\DeclareMathOperator{\supp}{supp} %

\newcommand{\prt}{\ensuremath{\partial}}

\newcommand{\pr}[1]{\ensuremath{\left( #1 \right) }}
\newcommand{\set}[1]{\ensuremath{\left\{ #1 \right\}}}

\newcommand{\norm}[1]{\ensuremath{\left\Vert #1 \right\Vert}}
\newcommand{\smallnorm}[1]{\ensuremath{\Vert #1 \Vert}}
\newcommand{\refS}[1]{Section~\ref{S:#1}}
\newcommand{\refT}[1]{Theorem~\ref{T:#1}}
\newcommand{\refL}[1]{Lemma~\ref{L:#1}}
\newcommand{\refP}[1]{Proposition~\ref{P:#1}}
\newcommand{\refD}[1]{Definition~\ref{D:#1}}
\newcommand{\refC}[1]{Corollary~\ref{C:#1}}
\newcommand{\refE}[1]{~(\ref{e:#1})}

\newcommand{\eps}{\ensuremath{\epsilon}}
\newcommand{\Cal}[1]{\ensuremath{\mathcal{#1}}}
\newcommand{\al}{\ensuremath{\alpha}}

\newcommand{\diff}[2]{\frac{ d#1}{d#2}}

\newcommand{\ol}{\overline}

\newcommand{\T}{\ensuremath{\Cal{T}}}

\newcommand{\W}{\ensuremath{\Cal{W}}}

\begin{document}

\raggedbottom

\numberwithin{equation}{section}

%
%
\newcommand{\MarginNote}[1]{
    \ifthenelse{\value{DetailLevel}=\value{dtlMarginComments} \or
            \value{DetailLevel}>\value{dtlMarginComments}} {
        \marginpar{
            \begin{flushleft}
                \footnotesize #1
            \end{flushleft}
            }
        }
        {}
    }

%
%
\newcommand{\NoteToSelf}[1]{
    }

%
%
\newcommand{\Obsolete}[1]{
    }

%
%
\newcommand{\Tentative}[1]{
    }

\newcommand{\Detail}[1]{
    \MarginNote{Detail}
    \skipline
    \hspace{+0.25in}\fbox{\parbox{4.25in}{\small #1}}
    \skipline
    }

\newcommand{\Todo}[1]{
    \skipline \noindent \textbf{TODO}:
    #1
    \skipline
    }

\newcommand{\Comment}[1] {
    \skipline
    \hspace{+0.25in}\fbox{\parbox{4.25in}{\small \textbf{Comment}: #1}}
    \skipline
    }

%
%

\newcommand{\IntTR}
    {\int_{t_0}^{t_1} \int_{\R^d}}

\newcommand{\IntAll}
    {\int_{-\iny}^\iny}

\newcommand{\Schwartz}
    {\ensuremath \Cal{S}}

\newcommand{\SchwartzR}
    {\ensuremath \Schwartz (\R)}

\newcommand{\SchwartzRd}
    {\ensuremath \Schwartz (\R^d)}

\newcommand{\SchwartzDual}
    {\ensuremath \Cal{S}'}

\newcommand{\SchwartzRDual}
    {\ensuremath \Schwartz' (\R)}

\newcommand{\SchwartzRdDual}
    {\ensuremath \Schwartz' (\R^d)}

\newcommand{\HSNorm}[1]
    {\norm{#1}_{H^s(\R^2)}}

\newcommand{\HSNormA}[2]
    {\norm{#1}_{H^{#2}(\R^2)}}

\newcommand{\Holder}
    {H\"{o}lder }

\newcommand{\Holders}
    {H\"{o}lder's }

\newcommand{\Holderian}
    {H\"{o}lderian }

\newcommand{\HolderRNorm}[1]
    {\widetilde{\Vert}{#1}\Vert_r}

\newcommand{\LInfNorm}[1]
    {\norm{#1}_{L^\iny(\Omega)}}

\newcommand{\SmallLInfNorm}[1]
    {\smallnorm{#1}_{L^\iny}}

\newcommand{\LOneNorm}[1]
    {\norm{#1}_{L^1}}

\newcommand{\SmallLOneNorm}[1]
    {\smallnorm{#1}_{L^1}}

\newcommand{\LTwoNorm}[1]
    {\norm{#1}_{L^2(\Omega)}}

\newcommand{\SmallLTwoNorm}[1]
    {\smallnorm{#1}_{L^2}}

\newcommand{\LpNorm}[2]
    {\norm{#1}_{L^{#2}}}

\newcommand{\SmallLpNorm}[2]
    {\smallnorm{#1}_{L^{#2}}}

\newcommand{\lOneNorm}[1]
    {\norm{#1}_{l^1}}

\newcommand{\lTwoNorm}[1]
    {\norm{#1}_{l^2}}

\newcommand{\MsrNorm}[1]
    {\norm{#1}_{\Cal{M}}}

\newcommand{\FTF}
    {\Cal{F}}

\newcommand{\FTR}
    {\Cal{F}^{-1}}

\newcommand{\InvLaplacian}
    {\ensuremath{\widetilde{\Delta}^{-1}}}

\newcommand{\EqDef}
    {\hspace{0.2em}={\hspace{-1.2em}\raisebox{1.2ex}{\scriptsize def}}\hspace{0.2em}}

%
%

%
%

\title
    [Vanishing viscosity for an expanding domain]
    {Vanishing viscosity limit for an expanding domain in space}

\author[Kelliher]{James P. Kelliher}
\address{Department of Mathematics, Brown University, Box 1917, Providence, RI
         02912}
\curraddr{Department of Mathematics, Brown University, Box 1917, Providence, RI
          02912}
\email{kelliher@math.brown.edu}

\author[Lopes Filho]{Milton C. Lopes Filho}
\address{IMECC, Caixa Postal 6065, University of Campinas -- UNICAMP, 13083-970, Campinas, SP, Brazil}
\curraddr{}
\email{mlopes@ime.unicamp.br}

\author[Nussenzveig Lopes]{Helena J. Nussenzveig Lopes}
\address{IMECC, Caixa Postal 6065, University of Campinas -- UNICAMP, 13083-970, Campinas, SP, Brazil}
\curraddr{} \email{hlopes@ime.unicamp.br}


\subjclass[2000]{Primary 76D05 ; Secondary 35Q30, 35Q35}

\keywords{Inviscid limit, Vanishing viscosity limit, Navier-Stokes equations,
Euler equations}

\date{} 


\begin{abstract}
We study the limiting behavior of viscous incompressible flows 
when the fluid domain is allowed to expand as the viscosity vanishes. We describe
precise conditions under which the limiting flow satisfies the full space Euler 
equations. The argument is based on truncation and on energy estimates, following
the structure of the proof of Kato's criterion for the vanishing viscosity limit.
This work complements previous work by the authors, see \cite{FLI2007,K2005UBD}.
\end{abstract}

\maketitle

%

\tableofcontents

%
%
\section{Introduction}\label{S:Introduction}

\noindent In \cite{FLI2007}, the second and third authors, in
collaboration with Drago\c{s} Iftimie, showed that, if an obstacle is
scaled by a factor $\eps$, then in the limit as viscosity vanishes the solutions to
the Navier-Stokes equations external to the obstacle converge strongly in
$L^\iny([0, T]; L^2)$ to a solution to the Euler equations in the whole space,
as long as $\eps < a \nu$ for a specific constant $a$. They also give the rate
of convergence in terms of $\nu$ and $\eps$.

In \cite{K2005UBD}, the first author considered the complementary problem of large domain
asymptotics, studying convergence to full plane flow of solutions of Euler or Navier-Stokes
in a large domain. The present article is a natural continuation of both \cite{FLI2007} 
and \cite{K2005UBD}. 

For a domain with boundary, it a classical open problem whether solutions of the Navier-Stokes
equations converge to solutions of the Euler equations when viscosity vanishes.
In \cite{FLI2007} the authors are considering two limits simultaneously: the
vanishing viscosity limit and the limit as the obstacle shrinks to a point,
solving the external problem for the Navier-Stokes equations. This means studying the way in which
a small boundary obstructs the vanishing viscosity convergence. Here, we consider what happens as a bounded domain expands by a factor $R$ to fill the whole space, giving the convergence rate in the vanishing viscosity limit for the internal problem in terms of $\nu$ and $R$. In the same spirit as \cite{FLI2007}, the present work regards 
the effect of distant boundaries in the vanishing viscosity limit.

More precisely, let $\Omega$ be a simply connected bounded domain in $\R^d$, $d = 2$ or $3$, with $C^2$-boundary $\Gamma$ and let $\Omega_R = R \Omega$ and $\Gamma_R = R \Gamma = \partial \Omega_R$, where we assume that the 
origin lies inside $\Omega$. 

A classical solution $(u, p)$
to the Euler equations without forcing in all of $\R^d$ satisfies
\begin{align*}
    \begin{matrix}
        (E) & \left\{
            \begin{array}{ll}
                \prt_t u + u \cdot \grad u + \grad p = 0
                    & \text{in } (0, T) \times \R^d, \\
                \dv u = 0
                    & \text{in } [0, T] \times \R^d, \\
                u = u_0
                    & \text{on } \set{0} \times \R^d,
            \end{array}
            \right.
    \end{matrix}
\end{align*}
where $\dv u^0 = 0$. A classical solution $(u^{\nu, R}, p^{\nu, R})$ to the
Navier-Stokes equations without forcing on $\Omega_R$ satisfies
\begin{align*}
    \begin{matrix}
        (NS) & \left\{
            \begin{array}{ll}
                \prt_t u^{\nu, R} + u^{\nu, R} \cdot \grad u^{\nu, R} + \grad p^{\nu, R}
                     = \nu \Delta u^{\nu, R}
                        & \text{in } (0, T) \times \Omega_R, \\
                \dv u^{\nu, R} = 0
                        & \text{in } [0, T] \times \Omega_R, \\
                u^{\nu, R} = 0
                    & \text {on } [0, T] \times \Gamma_R, \\
                u^{\nu, R} = u^{\nu, R}_0
                    & \text{on } \set{0} \times \Omega_R,
            \end{array}
            \right.
    \end{matrix}
\end{align*}
where $u^{\nu, R}_0 = 0$ on $\Gamma_R$. 

We will work, however, with weak solutions to the Navier-Stokes equations
(to avoid having to deal with the dependence of the time of existence of 
solutions on viscosity).

We consider the classical functions spaces,
\begin{align*}
    V(\Omega_R) &= \set{u \in H^1(\Omega_R) \colon
        \dv u = 0 \text{ in } \Omega_R \text{ and }
        u = 0 \text{ on } \Gamma_R}, \\
    H(\Omega_R) &= \set{u \in L^2(\Omega_R) \colon
        \dv u = 0 \text{ in } \Omega_R \text{ and }
        u \cdot \mathbf{n} = 0 \text{ on } \Gamma_R},
\end{align*}
where $\mathbf{n}$ is the outward directed unit normal vector field to $\Gamma_R$.

The spaces $V(\R^d)$ and $H(\R^d)$ are analogously defined. 

We define the space
\begin{align*}
    V_C(\R^d) &= \set{u \in V(\R^d) \colon
        \supp (\mbox{ curl } u) \text{ is compact}}.
\end{align*}
We will use the notation $\omega(u) \equiv \mbox{curl u}$ for the vorticity associated to a given velocity $u$.

In dimension two the condition $u \in V_C(\R^2)$  requires that the total mass of the vorticity be zero, see  Section 3.1.3 of \cite{MB2001} for a discussion. Hence, if we want to allow vorticities with distinguished sign we must allow for infinite
energy. To this end we recall the affine spaces $E_m$, introduced by J.-Y. Chemin in \cite{C1998}, following a 
construction by R. DiPerna and A. Majda, see \cite{DM1987}. We say that 
$u \in E_m$ if $u = v + \sigma$ for some $v \in H(\R^2)$ and for some stationary solution $\sigma$ of the Euler equations whose vorticity is smooth, compactly supported and has integral $m$. More precisely, for $\sigma$ given by  
\begin{equation} \label{sigma}
\sigma = \sigma(x) = \frac{x^{\perp}}{|x|^2} \int_0^{|x|} s\varphi(s)\,ds,
\end{equation}
for some $\varphi \in C^{\infty}_c(\real_+)$ and $2 \pi \int \varphi(s) s \, ds = m$. Given the arbitrariness in the choice of $\varphi$ we will assume, without loss of generality, that $\varphi$ is of distinguished sign. Above we used the notation $x^{\perp} = (-x_2,x_1)$ if 
$x = (x_1,x_2)$. Notice that $\omega(\sigma)(x) = \varphi(|x|)$.

The classical well-posedness results for weak solutions to ($E$) for $u_0$ in subspaces of $E_0 = H(\R^2)$ remain true when $E_0$ is replaced by $E_m$; see, for instance, Theorem 5.1.1 p. 85 of \cite{C1998} (Yudovich's theorem). In particular, fixing a value of $T > 0$, if $u_0$ is in $E_m \cap \dot{H}^1(\R^2)$ with compactly supported initial vorticity then the solution $u$ to ($E$) will lie in $C([0, T]; E_m) \cap L^\iny([0, T]; \dot{H}^1(\R^2))$. 

Throughout this paper we will assume that the initial velocity $u_0$ for solutions to ($E$) lies in $C^s(\R^d)$ for $s > 1$ so that a unique solution $u$ to the Euler equations ($E$) with initial velocity $u_0$ exists in the space $C^s([0, T] \times \R^d)$ for all $T < T^*$; see, for instance, Theorem 4.2.1 p. 77 of \cite{C1998} (or see \refT{EomegaBounds}, below). The time $T^*$ can be assumed to be arbitrary in two dimensions, where we also assume that $u_0$ lies in $E_m$ (see Theorem 4.2.4 p. 82 of \cite{C1998}), but only finite time existence is known in three dimensions. We assume that the initial vorticity is compactly supported with its support contained in a ball of radius $R_0$ and define
\begin{align*}
	R(T) = \inf_{r \ge 0} \set{r \colon \supp \omega(u) \subseteq [0, T] \times B_r(0)}.
\end{align*}
That $R(T)$ is finite follows in two dimensions from the transport of vorticity by the flow associated to $u$, $u$ being bounded uniformly over finite time. But $R(T)$ is also finite in three dimension, as we show in \refT{EomegaBounds}.
\begin{definition}[Classes of initial velocities]\label{D:Initu0} Let $s > 1$. 
    We treat the following three classes of initial velocities:
        \begin{enumerate}
            \item[I.]
                $u_0$ is in $C^s(\R^2) \cap V_C(\R^2)$,      

            \item[II.]
                $u_0$ is in $C^s(\R^2) \cap E_m \cap \dot{H}^1(\R^2)$, the support of $\omega(u_0)$ is compact,
                and $\Omega_1$ is a disk,

            \item[III.]
                $u_0$ is in $C^s(\R^3) \cap V_C(\R^3)$.
        \end{enumerate}
\end{definition}

We assume that the initial velocity $u^{\nu, R}_0$ is in $H(\Omega_R)$. For
such initial velocities it is a classical result of Leray that there exists a
weak solution $u^{\nu, R}$ to the Navier-Stokes equations ($NS$); in two dimensions
this solution is unique, a result due to Ladyzhenskaya. In three
dimensions, global-in-time existence is known, but not uniqueness, so we arbitrarily choose one such solution for each value of $\nu$. 

Our main result is the following:
\begin{theorem}\label{T:MainResult} Let $u_0$ be in one of the three classes of initial velocities in Definition \ref{D:Initu0} and set  
 $F(\nu,R) \equiv \smallnorm{u^{\nu, R}_0 - u_0}_{L^2(\Omega_R)}$.
    For all $T < T^*$ there exists a constant $C = C(s, T, \Omega, u_0) > 0$ such
    that
\begin{enumerate}
\item if $s > 1$,
    \begin{align*}
        \smallnorm{u^{\nu, R} - u}_{L^\iny([0, T]; L^2(\Omega_R))}
            \le \pr{C (\nu^{1/2} + R^{-\al}) + F(\nu,R)} e^{CT};
    \end{align*}
\item if $s \ge 2$,
    \begin{align*}
        \smallnorm{u^{\nu, R} - u}_{L^\iny([0, T]; L^2(\Omega_R))}
            \le \pr{C (\nu + R^{-\al}) + F(\nu,R)} e^{CT},
    \end{align*}
 \end{enumerate}   
for all sufficiently large $R$. 

The exponent $\al$ is 
    defined for each of the three cases as follows:
        \begin{enumerate}
            \item[I.]
               $\al = 1$,

            \item[II.]
                $\al = 1/3$,

            \item[III.]
                $\al = 1/2$.
        \end{enumerate}
\end{theorem}

Of particular interest is when we define $u^{\nu, R}_0$, independently of $\nu$, to be that unique divergence-free vector field tangent to the boundary of $\Omega_R$ whose vorticity on $\Omega_R$ is the same as that of $u_0$. In Section \ref{S:u0Truncation} such a vector field is denoted $u_0^{\nu,R}=\W_Ru_0$. We will see in Corollary \ref{C:WRProperties} that,   
\begin{align}\label{e:u0WDecay}
\smallnorm{u^{\nu, R}_0 - u_0}_{L^2(\Omega_R)}  =  F(\nu,R) \equiv F(R)  
        \le C R^{-\al}
\end{align}
for all $R \ge 2 R_0$, with $\al$ defined as in \refT{MainResult}. 
In this case, the term $F(R)$ in the bounds in \refT{MainResult} is dominated by the other term and so, in effect, it disappears.

It follows immediately from \refT{MainResult} that, as long as $R = R(\nu) \to
\iny$ as $\nu \to 0$ and $F(\nu,R) \to 0 $ as $R \to \iny$, $\smallnorm{u^{\nu, R(\nu)} - u}_{L^\iny([0, T];
L^2(\Omega_{R(\nu)}))} \to 0$ as $\nu \to 0$. 

It was shown in \cite{K2005UBD} for Case I that if $u^R$ is the
solution to the Euler equations on $\Omega_R$ with initial velocity $\T_R u_0$ then
$\smallnorm{u^R - u}_{L^\iny([0, T]; L^2(\Omega_{R(\nu)}))} \to 0$ as $R
\to \iny$. Here, $\T_R$ is a truncation operator, which will be defined precisely in \refS{uR}, see \eqref{e:TRu}. This result is extended in \cite{KStatSol2008} to cover Case II and to use the projector $P_{V(\Omega_R)}$---restriction to $\Omega_R$ followed by projection into $V(\Omega_R)$---in place of $\T_R$.
This gives the following corollary:
\begin{cor}\label{C:MainCor}
    Let $T < T^*$ and set $u^{\nu, R}_0 = P_{V(\Omega_R)} u_0$. Then, for $u_0$ as in Cases I or II,
    \begin{align*}
        \smallnorm{u^{\nu, R(\nu)} - u^{R(\nu)}}_{L^\iny([0, T]; L^2(\Omega_{R(\nu)}))}
            \to 0 \text{ as } \nu \to 0
    \end{align*}
    as long as $R = R(\nu) \to \iny$ as $\nu \to 0$.
\end{cor}

The energy argument in our proof of \refT{MainResult} follows fairly closely the argument in \cite{FLI2007}, which itself is closely connected to Kato's argument in \cite{Kato1983}.
We can describe in a unified way the approach of all three papers---\cite{Kato1983}, \cite{FLI2007}, and this one---as follows. Let $u^{NS}$ be the solution to ($NS$) in a domain $\Omega$ and let $u^E$ be the solution to ($E$) either in the whole space or, as in \cite{Kato1983}, in $\Omega$ itself. In \cite{Kato1983}, $\Omega$ is a fixed bounded domain; in \cite{FLI2007}, $\Omega$ is an external domain which is scaled to a point by a parameter $\eps$; for us, $\Omega$ is a bounded domain which is scaled by a parameter $R$ to fill the whole space.

Define a correction velocity $u^C$ to $u^E$ such that $u^C = u^E$ on $\prt
\Omega$ and is equal to zero outside a boundary layer $\Gamma_\delta$ of width
$\delta$. In \cite{Kato1983}, $\delta = C \nu$; in \cite{FLI2007}, $\delta = 
\eps$; in this paper, $\delta = C R^\al$. Let $u^A = u^E - u^C$ be an
``approximate solution'' to ($E$), and observe that $u^A = 0$ on $\prt \Omega$.

The goal is to bound the norm of $u^{NS} - u^E = u^{NS} - u^A - u^C$ in the
space $X = L^\iny([0, T]; L^2(\Omega))$. To do so, one first shows that
$\smallnorm{u^C}_X \to 0$ as $\delta \to 0$ or $\iny$ as the case may be. Then
one bounds $W = u^{NS} - u^A$ in $X$ by making an energy argument, the nature
of the argument differing in each case. Because $u^A = 0$ on $\prt \Omega$, no
troublesome boundary terms appear, though certain other terms appear because
$u^A$ is only an approximate solution to ($E$).

Kato's energy argument in \cite{Kato1983} is designed to estimate all of the
uncontrollable terms by the quantity
\begin{align}\label{e:KatosQuantity}
    \nu \int_0^T \norm{\grad u^{NS}}_{L^2(\Gamma_{C \nu})}^2,
\end{align}
which, by the most basic energy argument for solutions to
($NS$), is bounded uniformly for all $T$ and must vanish if the vanishing
viscosity limit is to hold. Kato's innovation is to show that the vanishing of this term is sufficient for the vanishing viscosity limit to hold.

\Ignore{
	The energy argument in \cite{FLI2007} employs an integration by parts that
	places all of the derivatives on $W$, $u^E$, or $u^A$, all of which can be
	controlled as long as the boundary shrinks sufficiently rapidly. The critical term 
	is $C \delta^{-1} \norm{W}_{L^2(A_\delta)}^2$, where $A_\delta$ is a bounded domain 
	that scales
	like $\delta$. An application of Poincare's inequality introduces a factor of
	$\delta^2$ which can be controlled as long as $\delta$ is proportional to
	$\nu$. In this paper, factors of $1/\delta = C R^{-\al}$ are advantageous, which
	simplifies the energy argument and allows convergence with no restrictions on
	the rate at which $R \to \iny$ as $\nu \to 0$.
	}

The results achieved in the three papers differ most fundamentally because for
Kato $\grad u^C$ scales like $1/\nu$, which is detrimental (but unavoidable),
introducing terms into the energy argument that cannot quite be controlled. For
us, $\grad u^C$ scales like $R^{-\al}$ which allows us to control all of these
terms. In \cite{FLI2007}, $\grad u^C$ scales like $1/\eps$, but the domain
shrinks in area like $\eps^2$, which largely counteracts the detrimental
effects of $\grad u^C$.

The research presented here is part of a series of papers aimed at studying asymptotic behavior of 
incompressible flows under singular domain perturbations. The first result in this line of research concerned 
ideal 2D flow in the exterior of a small obstacle, see \cite{FLI2003}, followed by a study of viscous 2D flow in the same
limit, see \cite{FLI2006}. Beyond these, this research has included ideal 2D flows in bounded domains with multiple holes one of which vanishes, see 
\cite{Lopes2006}, ideal or viscous 2D flow in a large domain, see \cite{K2005UBD}, 3D viscous flow in the exterior of a small obstacle, \cite{IK2008} and, most recently,  2D flow exterior to a smooth obstacle approaching a segment of a curve, see \cite{Lacave2008a} for the ideal flow case and \cite{Lacave2008b} for the viscous case. The classical open 
problem of vanishing viscosity in the presence of boundaries motivated the coupling of singularly perturbed domain problems with vanishing viscosity, specifically when the boundary disappears as viscosity vanishes. The first result in this direction was obtained in \cite{FLI2007} for the small 
obstacle limit and the current work can be regarded as a natural continuation of \cite{K2005UBD} in the same spirit.

This paper is organized as follows: Section 2 contains certain notation we use and conventions we follow. In \refS{uR} we describe an approximate solution $u^R$ to the Euler equations on $\Omega_R$ which we use in \refS{uProofOfMainResult} to prove \refT{MainResult}. The proof of \refT{MainResult} relies, however, on a long series of estimates involving $u^R$, which require us to understand how to take a divergence-free vector field defined in the whole plane or space and ``truncate'' it in such a way that it is unchanged in the central part of the domain $\Omega_R$, vanishes on the boundary of $\Omega_R$, and yet differs in the pertinent norms on $\Omega_R$ as little as possible from the original vector field. We describe the two dimensional version of such a truncation operator in \refS{Truncation2D} and use it in \refS{2DEstimates} to define and obtain the necessary estimates on $u^R$.

The definition and analysis of the truncation operator in three dimensions is markedly different from that in two dimensions. In \refS{Decay3D} we derive uniform-in-time bounds on the decay of the velocity and its gradient for a solution to ($E$). We then define the truncation operator in three dimensions in \refS{Truncation3D} and obtain the estimates on $u^R$ in three dimensions in \refS{3DEstimates}. In \refS{u0Truncation} we prove \refE{u0WDecay}. In \refS{Comments} we make some comments and state a couple of open problems.

%
%
\section{Preliminaries}\label{S:Notation}

\noindent The symbol $C$ stands for a positive constant that can hold different
values on either side of an inequality, though always has the same value on
each side of an equality.

For a scalar function $f$ in two dimensions we write $\grad^\perp f := (-\prt_2 f, \prt_1 f)$. In two dimensions we define the vorticity of a vector field $u$ to be the scalar curl, $\omega = \omega(u) := \prt_1 u^2 - \prt_2 u^1 \equiv \nabla^\perp\cdot u$. In three dimensions, we define the vorticity to be $\omega = \omega(u) := \curl u$; that is, $\omega$ is the three-vector,
\begin{align*}
	\omega = (\prt_2 u^3 - \prt_3 u^2, \prt_3 u^1 - \prt_1 u^3, \prt_1 u^2 - \prt_2 u^1).
\end{align*}

 It is sometimes convenient in three dimensions to view the vorticity as the anti-symmetric $3 \times 3$ matrix $\mathbb{A} = 
\mathbb{A}(u)$ whose entry in the $i$-th row, $k$-th column is  $\omega^i_k =  \omega^i_k(u) := (\prt_k u^i - \prt_i u^k)/2$. Thus,
\begin{align*}
	\mathbb{A} = \frac{1}{2}
		\begin{pmatrix}
			0 & -\omega^3 & \omega^2 \\
			\omega^3 & 0 & -\omega^1 \\
			-\omega^2 & \omega^1 & 0
		\end{pmatrix}.
\end{align*}
Observe that the $L^p$-norms of $\mathbb{A}$ and $\omega$ are equivalent, differing only by a multiplicative constant.

Given a divergence-free $C^1$ vector field $u$ on $\real^2$ let $\omega=\omega(u)$ be its vorticity, which we assume to have compact support. We define the associated two-dimensional stream function $\psi$ as 
\begin{equation} \label{2Dstream}
\psi=\psi(x)=\frac{1}{2\pi}\int_{\real^2}\log|x-y| \omega(y)\,dy,
\end{equation}
so that $\Delta \psi = \omega$ and $u = \nabla^{\perp}\psi$. 

Given a divergence-free $C^1$ vector field $u$ on $\real^3$  with compactly supported vorticity $\omega=\omega(u)$ we define the associated three-dimensional (vector-valued) stream function $\Psi$ as 
\begin{equation} \label{3Dstream}
\Psi = \Psi(x) = \frac{1}{4\pi}\int_{\real^3}\frac{1}{|x-y|}\omega(y)\,dy.
\end{equation}
Hence, $-\Delta \Psi = \omega$ and $u = \mbox{ curl } \Psi$, the latter statement following since $\mbox{div }\Psi = 0$, which in turn can be seen from the equation $\Delta \mbox{div } \Psi = 0$, $\mbox{ div }\Psi \to 0$ at $\infty$.

We note in passing that an alternative to this vector-valued stream function is to define the matrix-valued stream function  
\[\psi_{ik} := \frac{1}{2\pi}\int_{\real^3} \frac{1}{|x-y|} \omega^i_k(y)\,dy,\]
which has the property that $u^i = \sum_k \partial_k \psi_{ik}$. The advantage of defining the stream function in this way is that it can be generalized to higher dimensions.

%
%
\section{Approximate solution to the Euler equations}\label{S:uR}

\noindent Define a cutoff function $\varphi^R$ in two dimensions as follows. Fix $\theta$ in $[0, 1]$. (We will ultimately choose a value of $\theta$ that optimizes the convergence rate in \refT{MainResult}.) Let $\delta_1 = 1/2\ol{\kappa}$, where $\ol{\kappa}$ is the maximum curvature of $\Gamma = \partial \Omega$. Let $\Sigma_R$ be a tubular neighborhood of $\Gamma_R$ in $\Omega_R$ of uniform width $\delta_1 R^\theta$ for all $R$ in $[1, \iny)$. (Decrease the value of $\delta_1$ if necessary to insure that the origin is not contained in $\Sigma_R$.) Put coordinates $(s, r)$ on $\Sigma_R$, where $s$ is arc length along $\Gamma$, which locates a point on $\Gamma$, and $r$ is the distance along the inward normal at that point.

Let $g$ in $C^\iny([0, \delta_1])$ taking values in $[0, 1]$ be defined so that
$g(0) = g'(0) = 0$ and $g = 1$ on $\Big[\displaystyle{\frac{\delta_1}{2}}, \delta_1\Big]$. Then define $\varphi^R$
in $C^\iny(\Omega_R)$ by $\varphi^R(s, r) = g(R^{-\theta} r)$ for points $(s, r)$ in $\Sigma_R$, and
$\varphi^R = 1$ on $\Omega_R \setminus \Sigma_R$. Observe that
\begin{align}\label{e:varphiBounds}
    \begin{split}
        &\norm{\grad \varphi^R}_{L^\iny(\Omega_R)} \le C R^{-\theta}, \qquad
            \norm{\grad \grad \varphi^R}_{L^\iny(\Omega_R)} \le C R^{-2 \theta},
   \end{split}
\end{align}
and similarly for higher derivatives of $\varphi^R$, where $C$ is independent of $R$ in $[1, \iny)$, and $\varphi^R = 0$ and $\grad \varphi^R = 0$ on $\Gamma$.

We define $\varphi^R$ in three dimensions more simply. Let
\begin{align*}
    \Sigma = \set{x \in \Omega: \dist(x, \Gamma) < 1/2 \ol{\kappa}},
\end{align*}
where $\ol{\kappa}$ is the maximum of all sectional curvatures over all points of $\Gamma$. Let $\varphi$ in $C^\iny(\Omega)$ taking values in $[0, 1]$ be defined so that $\varphi = 1$ on $\Omega \setminus \Sigma$ and $\varphi = 0$, $\grad \varphi = 0$ on $\Gamma$, and let $\varphi^R(\cdot) = \varphi(\cdot/R)$ and $\Sigma_R = R \Sigma$. Then \refE{varphiBounds} holds with $\theta = 1$.

Let $\psi$ be the two-dimensional stream function associated to the full-plane Euler velocity $u$, as in \eqref{2Dstream}. We define the vector field 
$u^R$ on $\Omega_R$ by
\begin{align}\label{e:TRu}
    u^R = \T_R u
        := \grad^\perp (\varphi^R \psi).
\end{align}
Notice that this defines an operator $\T_R$ whose properties we will explore later.

If $\Psi$ is the three-dimensional stream function associated to the full-space Euler velocity $u$, as in \eqref{3Dstream}, then we define the approximation $u^R$ on $\Omega_R$ by
\begin{align}\label{e:TRu3D}
    u^R = \T_R u
        := \grad \times (\varphi^R \Psi).
\end{align}

The operator $\T_R$ in both cases has the property that $u^R = \T_R u$ lies not just in $H(\Omega_R)$ but in $V(\Omega_R)$, and so vanishes on the boundary. It also satisfies ($E$) in $\Omega_R \setminus \Sigma_R$. In this sense, it is an approximate solution to ($E$). 

Clearly, $u^R$ satisfies the identity
\begin{align}\label{e:uRIdentity2D}
    \prt_t u^R
       &= - \varphi^R u \cdot \grad u
          - \varphi^R \grad p
          + \prt_t \psi \grad^\perp \varphi^R
\end{align}
in two dimensions and  
\begin{align}\label{e:uRIdentity3D}
    \prt_t u^R
       &= - \varphi^R u \cdot \grad u
          - \varphi^R \grad p
          + \grad \varphi^R \times \prt_t \Psi,
\end{align}
in three dimensions.

Next we state a proposition which contains the key estimates on $u^R$ that we will use in \refS{uProofOfMainResult} to prove \refT{MainResult}. 
We prove the two-dimensional case of this proposition in \refS{2DEstimates} and the three-dimensional case in \refS{3DEstimates}.

\begin{prop}\label{P:FLILikeBasicEstimates}
 For all $T < T^*$, for all sufficiently large $R$, we have 
    \begin{enumerate}
        \item[(1)]
            $\smallnorm{\grad u^R}_{L^\iny([0, T]; L^2(\Omega_R))} 
            	\le C$,

        \item[(2)]
            $\smallnorm{u^R}_{L^\iny([0, T] \times \Omega_R)} 
            	\le C$,

        \item[(3)]
            $\smallnorm{\grad u^R}_{L^\iny([0, T] \times \Omega_R)}
            	\le C$,

        \item[(4a)]
            $\smallnorm{p \grad \varphi^R}_{L^\iny([0, T]; L^2(\Omega_R))}
                + \smallnorm{\prt_t \psi \grad \varphi^R}_{L^\iny([0, T]; 
                		L^2(\Omega_R))}$
		
			$\qquad\qquad \le C R^{-\theta}$ in 2D,

        \item[(4b)]
            $\smallnorm{p \grad \varphi^R}_{L^\iny([0, T]; L^2(\Omega_R))}
                + \smallnorm{\grad \varphi^R \times  
                	\prt_t \Psi}_{L^\iny([0, T]; L^2(\Omega_R))}$
	
			$\qquad\qquad \le C R^{-1}$ in 3D,

        \item[(5)]
            $\smallnorm{\Delta u^R}_{L^\iny([0, T]; L^2(\Omega_R))} \le C$ when $s \ge 2,$

        \item[(6)]
            $\smallnorm{u^R - u}_{L^\iny([0, T]; L^2(\Omega_R))}
                + \smallnorm{u^R - \varphi^R u}_{L^\iny([0, T]; L^2(\Omega_R))}
            	\le C R^{-\al}$,

        \item[(7)]
             $\smallnorm{\grad(u - u^R)}_{L^\iny([0, T]; L^2(\Omega))} 
             	\le C R^{-\beta}$.
    \end{enumerate}
Above, $\alpha$ and $\beta$ are given by:	
	\begin{align}\label{e:alDef}
		\al &=
			\left\{
				\begin{array}{ll}
					1/2 + \theta/2 
						& \text{if } m = 0, \\
					1/2 - \theta/2
						& \text{if } m \ne 0 \\
				\end{array}
				\right.
    \end{align}
    and
	\begin{align}\label{e:betaDef}
		\beta &=
			\left\{
				\begin{array}{ll}
					1/2 + 3 \theta/2 
						& \text{if } m = 0, \\
					1/2 + \theta/2
						& \text{if } m \ne 0, \\
				\end{array}
				\right.
    \end{align}
in two dimensions, while  $\al = 1/2$ and $\beta = 3/2$ in three dimensions. 

    For case I of \refD{Initu0} the constants above depend only on $\Omega$;
    for cases II and III some of the constants also depend on $T$.
\end{prop}

%
%
\section{Energy argument}\label{S:uProofOfMainResult}

\noindent \begin{proof}[\textbf{Proof of \refT{MainResult}}] The proof proceeds much as in Section 2 of \cite{FLI2007}: Using our approximate solution 
$u^R$ to ($E$) we make an energy argument to bound the difference
\begin{align*}
	W = u^{\nu, R} - u^R
\end{align*}
in the $L^2$ norm. Then using inequality (6) of \refP{FLILikeBasicEstimates} we apply the triangle inequality to complete the proof. 
We give the argument in $2$ dimensions only; it is valid with minor adaptations in $3$ dimensions. The only delicate point in adapting 
to $3$ dimensions is that we deal with weak Leray solutions for which we cannot perform energy estimates. However the energy inequality is equivalent to the necessary estimates; see \cite{FLI2007} for a more detailed discussion of this issue.

Subtracting the identity in \refE{uRIdentity2D} from ($NS$) we obtain
\begin{align*}
    \prt_t W - \nu \Delta W
        = - u^{\nu, R} \cdot &\grad u^{\nu, R}
          - \grad p^{\nu, R}
          + \nu \Delta u^R
          + \varphi^R u \cdot \grad u \\
          &+ \varphi^R \grad p
          - \prt_t \psi \grad^\perp \varphi^R.
\end{align*}
Multiplying both sides by $W$ and integrating over $\Omega_R$ gives
\begin{align*}
    \frac{1}{2} &\diff{}{t} \smallnorm{W}_{L^2}^2 + \nu \norm{\grad W}_{L^2}^2
        = I_1 + I_2 + I_3 + I_4 + I_5,
\end{align*}
where when $s > 1$,
\begin{align*}
    \begin{array}{ll}
        I_1 = - \nu \int_{\Omega_R} \grad W \cdot \grad u^R,
       &I_2 = - \int_{\Omega_R} (u^{\nu, R} \cdot \grad u^{\nu, R}) \cdot W, \\
        I_3 = \int_{\Omega_R} (\varphi^R u \cdot \grad u) \cdot W,
       &I_4 = \int_{\Omega_R} \varphi^R \grad p \cdot W, \\
        I_5 = - \int_{\Omega_R} \prt_t \psi \grad^\perp \varphi^R \cdot W.
    \end{array}
\end{align*}
In $I_1$ we integrated by parts to remove $\Delta u^R$, but when $s > 2$ it is more
advantageous to retain it, using
\begin{align*}
    I_1 = \nu \int_{\Omega_R} W \cdot \Delta u^R.
\end{align*}

When $s > 1$ we apply the Cauchy-Schwarz and Young's inequalities to the first form of $I_1$ to get
\begin{align*}
    \abs{I_1}
       &\le \frac{\nu}{2} \pr{\norm{\grad W}_{L^2}^2 + \smallnorm{\grad u^R}_{L^2}^2}
        \le \frac{\nu}{2} \norm{\grad W}_{L^2}^2 + C \nu,
\end{align*}
and when $s > 2$ we apply the Cauchy-Schwarz inequality to the second form of $I_1$ to get
\begin{align*}
    \abs{I_1}
       &\le C \nu \smallnorm{\Delta u^R}_{L^2(\Omega_R)} \norm{W}_{L^2(\Omega_R)}
        \le C \nu \norm{W}_{L^2(\Omega_R)}.
\end{align*}

Summing $I_2$ and $I_3$ and using
\begin{align*}
    \int_{\Omega_R} (u^{\nu, R} \cdot \grad W) \cdot W = 0
\end{align*}
we have
\begin{align*}
    &\abs{I_2 + I_3} \\
       &\quad= \abs{\int_{\Omega_R} (u^R \cdot \grad (u - u^R)) \cdot W
                - (W \cdot \grad u^R) \cdot W
                + [(\varphi^R u - u^R) \cdot \grad u] \cdot W} \\
       &\quad\le \smallnorm{u^R}_{L_\iny} \norm{\grad (u - u^R)}_{L^2}
                    \norm{W}_{L^2}
                + \smallnorm{\grad u^R}_{L^\iny} \norm{W}_{L^2}^2 \\
       &\qquad\qquad\qquad
                + \smallnorm{\varphi^R u - u^R}_{L^2}
                  \smallnorm{\grad u}_{L^\iny}
                  \norm{W}_{L^2} \\
       &\quad\le C R^{-\beta} \norm{W}_{L^2} + C \norm{W}_{L^2}^2 + C R^{-\al}
                  \norm{W}_{L^2} \\
       &\qquad\le C (R^{-\al} + \norm{W}_{L^2}) \norm{W}_{L^2},
\end{align*}
where we used inequalities (6) and (7) from \refP{FLILikeBasicEstimates} and also that $\al \le \beta$.

Summing $I_4$ and $I_5$ and integrating the first term by parts gives,
\begin{align*}
	\abs{I_4 + I_5}
        	&= \abs{\int_{\Omega_R} p \grad \varphi^R \cdot W +
               \prt_t \psi \grad^\perp \varphi^R \cdot W} \\
       &\le \pr{\smallnorm{p \grad \varphi^R}_{L^2}
                + \smallnorm{\prt_t \psi \grad^\perp \varphi^R}_{L^2}}
                  \norm{W}_{L^2}
        \le C R^{-\theta} \norm{W}_{L^2} \\
       &\le C R^{-2 \theta} + C \norm{W}_{L^2}^2.
\end{align*}

When $s > 1$, we conclude that
\begin{align*}
    \frac{1}{2} &\diff{}{t} \smallnorm{W}_{L^2}^2 + \nu \norm{\grad W}_{L^2}^2 \\
       &\qquad
       	\le \frac{\nu}{2} \norm{\grad W}_{L^2}^2
            + C \nu + C R^{-\al} \norm{W}_{L^2} + C R^{-2 \theta} + C \norm{W}_{L^2}^2 \\
       &\qquad
       	\le \frac{\nu}{2} \norm{\grad W}_{L^2}^2
            + C \nu + C (R^{-2 \al} + R^{-2 \theta}) + C \norm{W}_{L^2}^2.
\end{align*}
Integrating in time gives
\begin{align*}
    &\smallnorm{W(t)}_{L^2}^2 + \nu \int_0^t \norm{\grad W}_{L^2}^2 \\
       &\qquad
       	\le \smallnorm{W(0)}_{L^2}^2
            + C T \nu + C T (R^{-2 \al} + R^{-2 \theta}) + C \int_0^t \norm{W}_{L^2}^2.
\end{align*}

It follows from Gronwall's lemma that
\begin{align*}
    &\norm{W}_{L^\iny([0, T]; L^2(\Omega_R))} 
    	\le \pr{F(R)^2 + C (\nu + R^{-2 \al} + R^{-2 \theta})}^{1/2} e^{CT} \\
		&\qquad
			\le \pr{F(R) + C (\nu^{1/2} + R^{-\al} + R^{-\theta})} e^{CT}.
\end{align*}

Then from the triangle inequality and inequality (6) of \refP{FLILikeBasicEstimates},
\begin{align*}
   &\smallnorm{u^{\nu, R} - u}_{L^\iny([0, T]; L^2(\Omega_R))} \\
        &\qquad
            \le \smallnorm{u^R - u}_{L^\iny([0, T]; L^2(\Omega_R))}
            + \norm{W}_{L^\iny([0, T]; L^2(\Omega_R))} \\
        &\qquad
            \le C R^{-\al} + \pr{F(R) + C (\nu^{1/2} + R^{-\al} + R^{-\theta})} e^{CT} \\
        &\qquad
            \le \pr{F(R) + C (\nu^{1/2} + R^{-\al} + R^{-\theta})} e^{CT}.
\end{align*}

When $s > 2$, we have instead that
\begin{align*}
    \smallnorm{W}_{L^2} &\diff{}{t} \smallnorm{W}_{L^2}
        = \frac{1}{2} \diff{}{t} \smallnorm{W}_{L^2}^2
        \le \frac{1}{2} \diff{}{t} \smallnorm{W}_{L^2}^2
                    + \nu \norm{\grad W}_{L^2}^2 \\
       &\le C \nu \norm{W}_{L^2}
            + C (R^{-\al} + R^{-\theta}) \norm{W}_{L^2} + C \norm{W}_{L^2}^2.
\end{align*}
Dividing both sides by $\smallnorm{W}_{L^2}$ (it is easy to see that division by zero will not invalidate the following inequality after integrating in time) gives
\begin{align*}
    \diff{}{t} \smallnorm{W}_{L^2}
        \le C \nu
            + C R^{-\al} + C R^{-\theta} + C \norm{W}_{L^2}.
\end{align*}
Integrating in time and applying Gronwall's lemma, we have
\begin{align*}
    \smallnorm{W}_{L^2}
        \le \pr{F(R) + C (\nu + C R^{-\al} + R^{-\theta})} e^{CT}.
\end{align*}
The bound on $\smallnorm{u^{\nu, R} - u}_{L^\iny([0, T]; L^2(\Omega_R))}$
follows from the triangle inequality as for $s > 1$.

The value of $\al$ in the statement of \refT{MainResult} is chosen so that $\alpha=\alpha(\theta)$ gives the optimal rate of convergence in each case;  this corresponds to $\theta = 1$ for Case I; $\theta = 1/3$ for Case II so that $\theta = \al$; and $\theta = 1$ was fixed for Case III. 
\end{proof}

%
%
\section{Truncation operator in 2D}\label{S:Truncation2D}

\noindent Let $u$ be in $E_m \cap C^1$ for some $m$ in $\R$ with vorticity $\omega(u)$ having compact support in a ball of radius 
$R_0$.  Let $\psi$ be the stream function, as defined by the expression in \eqref{2Dstream}.

Let $\varphi^R$ and $\Sigma_R$ be defined as in \refS{uR} and  recall the definition of $u^R$ and $\T_R$ given in \eqref{e:TRu}, 
$u^R=\T_R u = \nabla^{\perp}(\varphi^R\psi)$. 
To explore the properties of $\T_R$ we must first establish some bounds on the $L^2$ norms of $u$, $\grad u$, and $\psi$ in $\Sigma_R$. To this end we write 
\[u = v + \sigma,\]
where $v$ is in $V_C(\real^2)$ and $\sigma$ is a stationary solution with radially symmetric, smooth, compactly supported vorticity of integral $m$; we assume that $\omega(\sigma)$ is of distinguished sign. We can assume, without loss of generality, that the support of $\omega(\sigma)$ is also contained in the ball of radius $R_0$, from which it follows that the support of $\omega(v)$ is contained in this same ball. Now, $v$ and $\sigma$ are also $C^1$ divergence-free vector fields and hence we can define their associated stream functions $\psi_v$ and $\psi_{\sigma}$ using the expression in \eqref{2Dstream}. But then $v = \nabla^{\perp}\psi_v$ and $\sigma = \nabla^{\perp} \psi_{\sigma}$. It follows in particular that $v$ can be written in terms of $\omega(v)$ through the Biot-Savart law 
$v = K \ast \omega(v)$, an integral operator with kernel
\begin{equation} \label{BSkernel2D}
K=K(z) = \frac{1}{2\pi}\frac{z^{\perp}}{|z|^2}.
\end{equation}
From this explicit expression and using the fact that the integral of $\omega(v)$ vanishes, together with the easily obtained estimate
\[ \smallnorm{\omega(v)}_{L^1(\R^2)}\leq 2\smallnorm{\omega(u)}_{L^1(\R^2)},\] 
it follows that there exists $C = C(R_0)>0$ such that  
\begin{align}\label{e:vGrowthBound}
	\begin{split}
	    \abs{v(x)} 
	    	&\le C C_0/\abs{x}^2, \;
	    \abs{\grad v(x)} 
	    	 \le C C_0/\abs{x}^3
	\end{split}
\end{align}
for all $\abs{x} \ge 2R_0$, with $C_0 =  2\smallnorm{\omega(u)}_{L^1(\R^2)}$. 
Similarly, it follows from the explicit expression for $\psi_v$, \eqref{2Dstream}, that 
 \begin{align}\label{e:psivBound}
    \abs{\psi_v(x)} \le C C_0/\abs{x}
        \text{ for all } \abs{x} \ge 2R_0.
\end{align}

Put coordinates on $\Sigma_R$ as in the definition of $\varphi_R$ in \refS{uR}. Letting $a$ be the length of $\Gamma_1$ it follows that the length of 
$\Gamma_R$ is $aR$. Then
\begin{align*}
	&\smallnorm{\psi_v}_{L^2(\Sigma_R)}^2
		= \int_0^{aR} \int_0^{\delta_1 R^\theta}
			\abs{J(s, r)} \abs{\psi_v(s, r)}^2 \, dr \, ds,
\end{align*}
where $J$ is the Jacobian of the transformation from rectangular coordinates to $(s, r)$-coordinates. Because of the way we constructed $\Sigma_R$ and because $\theta \le 1$, $\abs{J} \le C$ and $\Sigma_R$ lies outside a ball of radius $C(\Omega) R$.
Thus by \refE{psivBound}, $\abs{\psi_v(s, r)} \le C C_0/R$ in the integral above as long as $C(\Omega) R \ge 2 R_0$; that is, as long as
\begin{align}\label{e:MuCondition}
	R \ge \mu(\Omega) R_0,
\end{align}
where $\mu(\Omega) = 2/C(\Omega)$ depends only on the geometry of $\Omega$. Then
\begin{align*}
	&\smallnorm{\psi_v}_{L^2(\Sigma_R)}
		\le \pr{C C_0^2 R^{-2} aR \delta_1 R^\theta}^{1/2}
		\le C C_0 R^{\theta/2 - 1/2}.
\end{align*}
Since each of $\grad^\perp$ and $\grad$ introduces an extra factor of $1/\abs{x}$, it follows that
\begin{align*}
	\smallnorm{v}_{L^2(\Sigma_R)}
		\le C C_0 R^{\theta/2 - 3/2}, \;
	\smallnorm{\grad v}_{L^2(\Sigma_R)} 
		\le C C_0 R^{\theta/2 - 5/2}.
\end{align*}

As previously pointed out, see \eqref{sigma},  $\sigma$ is given by 
\begin{align*}
	\sigma(x_1, x_2) 
		&= \pr{-\frac{x_2}{\abs{x}^2} \int_0^{\abs{x}} r \omega(\sigma)(r) \, dr, \,
		       \frac{x_1}{\abs{x}^2} \int_0^{\abs{x}} r \omega(\sigma)(r) \, dr}
\end{align*}
so that $\abs{\sigma(x)} = \abs{m}(2 \pi)^{-1}/\abs{x}$ for $\abs{x} \ge R_0$.   Thus, $\sigma$ decays like $\psi_v$ so we can see that
\begin{align*}
	&\smallnorm{\sigma}_{L^2(\Sigma_R)}
		\le C \abs{m} R^{\theta/2 - 1/2}, \;
	\smallnorm{\grad \sigma}_{L^2(\Sigma_R)}
		\le C \abs{m} R^{\theta/2 - 3/2}.
\end{align*}

The expression for $\psi_{\sigma}$ can be calculated directly using \eqref{2Dstream} together with the radial symmetry of $\omega(\sigma)$. Of course, we can add an arbitrary constant to $\psi_{\sigma}$ and still satisfy the equations $\sigma = \nabla^{\perp}\psi_{\sigma}$ and $\Delta \psi_{\sigma} = \omega(\sigma)$. For $|x| \geq R_0$ we obtain:
\[\psi_{\sigma}(x)= \frac{m}{2\pi} \log |x| + C.\]

Since when $m \ne 0$ we assume that $\Omega$ is a disk centered at the origin, we can choose the constant $C_R$ so that $\psi_\sigma = 0$ on 
$\Gamma_R$. The  value of $\prt_t \psi$ is unaffected by the choice of $C_R$, however, and $\prt_t \psi$ in inequality (4a) of \refP{FLILikeBasicEstimates} is the only direct use of $\psi$ that we make, so the choice of $C_R$, though it depends on $R$, will not affect any of our estimates. 

Applying Poincare's inequality (or integrating $\psi_\sigma$ directly) gives
\begin{align*}
	\smallnorm{\psi_\sigma}_{L^2(\Sigma_R)}
	   &\le C R^\theta \smallnorm{\sigma}_{L^2(\Sigma_R)}
	    \le C \abs{m} R^{3 \theta/2 - 1/2}.
\end{align*}
The factor of $R^\theta$ here comes from the thickness of $\Sigma_R$.

Adding the corresponding bounds for $u = v + \sigma$ and $\psi = \psi_v + \psi_\sigma$, 
\begin{align}\label{e:psiuL2SigmaRBound}
	&\smallnorm{\psi}_{L^2(\Sigma_R)}
		\le C C_0 R^{\theta/2 - 1/2} + C \abs{m} R^{3 \theta/2 - 1/2},
\end{align}
and for the velocity
\begin{align}\label{e:uL2SigmaRBound}
	\smallnorm{u}_{L^2(\Sigma_R)}
		\le C C_0 R^{\theta/2 - 3/2} + C \abs{m} R^{\theta/2 - 1/2}
\end{align}
and
\begin{align}\label{e:graduL2SigmaRBound}
	\smallnorm{\grad u}_{L^2(\Sigma_R)} 
		\le C C_0 R^{\theta/2 - 5/2} + C \abs{m} R^{\theta/2 - 3/2}.
\end{align}
These inequalities each hold as long as \refE{MuCondition} holds.

Let $X$ be the subspace of all vector fields in $E_m \cap \dot{H}_1(\R^2)$ whose vorticity has compact support.
We can now describe the relevant properties of the two dimensional truncation operator, adapting Lemma 4.2 of \cite{K2005UBD}.
\begin{prop}\label{P:Truncate2D}
    Let $\Omega$ be a disk centered at the origin and let the truncation operator
    $\T_R$ be defined as in \refE{TRu}.
    Then $\T_R \colon X \to V(\Omega_R)$ with the
    property that for all $u$ in $X$
    \begin{align}\label{e:uTruL2Bound}
        \smallnorm{u - \T_R u}_{L^2(\Omega_R)} 
        	\le C\smallnorm{\omega(u)}_{L^1(\R^2)} R^{-\al}
	\end{align}
    and
    \begin{align}\label{e:ugradTruL2Bound}
        \smallnorm{\grad(u - \T_R u)}_{L^2(\Omega_R)} 
        	\le C\smallnorm{\omega(u)}_{L^1(\R^2)} R^{-\beta}
	\end{align}
    for all $R$ satisfying \refE{MuCondition}, where $C = C(\Omega)$,
    $\al$ is defined in \refE{alDef}, and $\beta$ is defined in \refE{betaDef}.
    Also,
    \begin{align}\label{e:TRuL2Bound}
		\smallnorm{\grad \T_R u}_{L^2(\Omega_R)}
			\le C \smallnorm{\omega(u)}_{L^2(\R^2)},
	\end{align}
	where $C = C(R_0)$.
\end{prop}
\begin{proof}
	Using the inequalities in 
    \refE{varphiBounds}, \refE{psiuL2SigmaRBound}, \refE{uL2SigmaRBound}, and 
    \refE{graduL2SigmaRBound}, we have,
    \begin{align*}
        &\smallnorm{u - \T_R u}_{L^2(\Omega_R)}
                = \smallnorm{\grad^\perp \psi - \varphi^R \grad^\perp \psi
                            - \psi \grad^\perp \varphi^R}_{L^2(\Omega_R)} \\
        &\qquad\le \smallnorm{(1 - \varphi^R) \grad^\perp \psi}_{L^2(\Omega_R)}
                + \smallnorm{\psi \grad^\perp \varphi^R}_{L^2(\Omega_R)} \\
        &\qquad\le \smallnorm{u}_{L^2(\Sigma_R)}
                + \smallnorm{\grad\varphi^R}_{L^\iny(\Sigma_R)}
                \smallnorm{\psi}_{L^2(\Sigma_R)} \\
        &\qquad\le C C_0 R^{\theta/2 - 3/2} + C \abs{m} R^{\theta/2 - 1/2}
       			+ C R^{-\theta} C_0 R^{\theta/2 - 1/2} \\
        	&\qquad\qquad\qquad\qquad
				+ C R^{-\theta} \abs{m} R^{3\theta/2 - 1/2} \\
        &\qquad\le C C_0 R^{\theta/2 - 3/2} + C \abs{m} R^{\theta/2 - 1/2}
       			+ C C_0 R^{-\theta/2 - 1/2} \\
        	&\qquad\qquad\qquad\qquad
				+ C \abs{m} R^{\theta/2 - 1/2} \\
		&\qquad\le C C_0 R^{-\theta/2 - 1/2} + C \abs{m} R^{\theta/2 - 1/2}.
    \end{align*}
    In the final inequality we needed to only include the two terms that
    dominate (depending on whether 
    $m = 0$) for all $\theta$ in $[0, 1]$. Clearly, $|m| \leq \smallnorm{\omega(u)}_{L^1(\R^2)}$, so this gives 
    \refE{uTruL2Bound}.

    Similarly,
    \begin{align*}
        &\smallnorm{\grad u - \grad \T_R u}_{L^2(\Omega_R)}
            = \smallnorm{\grad u - \grad \grad^\perp (\varphi^R \psi)}_{L^2(\Omega_R)} \\
            &\qquad= \smallnorm{\grad u - \grad(\varphi^R \grad^\perp \psi)
                    - \grad(\psi \grad^\perp \varphi^R)}_{L^2(\Omega_R)} \\
            &\qquad = \smallnorm{\grad u - \varphi^R \grad \grad^\perp \psi
                        - \grad \varphi^R \otimes \grad^\perp \psi
                    - \grad \psi \otimes \grad^\perp \varphi^R
                    - \psi \grad \grad^\perp \varphi^R}_{L^2(\Omega_R)} \\
            &\qquad = \smallnorm{(1 - \varphi^R) \grad u
                        - \grad \varphi^R \otimes \grad^\perp \psi
                    - \grad \psi \otimes \grad^\perp \varphi^R
                    - \psi \grad \grad^\perp \varphi^R}_{L^2(\Omega_R)} \\
            &\qquad \le \smallnorm{\grad u}_{L^2(\Sigma_R)}
                + 2 \norm{\grad \varphi^R}_{L^\iny(\Sigma_R)}
                    \smallnorm{u}_{L^2(\Sigma_R)}
                + \smallnorm{\grad \grad \varphi^R}_{L^\iny(\Sigma_R)}
                    \norm{\psi}_{L^2(\Sigma_R)} \\
            &\qquad \le C C_0 R^{\theta/2 - 5/2} + C \abs{m} R^{\theta/2 - 3/2} 
                    + C R^{-\theta} \pr{C_0 R^{\theta/2 - 3/2} 
                    + \abs{m} R^{\theta/2 - 1/2}} \\
            &\qquad\qquad
                    + C R^{-2 \theta} \pr{ C_0 R^{\theta/2 - 1/2} 
                    	+ \abs{m} R^{3 \theta/2 - 1/2}} \\
            &\qquad
            	= C C_0 R^{\theta/2 - 5/2} + C \abs{m} R^{\theta/2 - 3/2} 
                    + C C_0 R^{-\theta/2 - 3/2} + C \abs{m} R^{-\theta/2 - 1/2} \\
            &\qquad\qquad
                    + C C_0 R^{-3 \theta/2 - 1/2} + C \abs{m} R^{-\theta/2 - 1/2} \\
            &\qquad
            	\le C C_0 R^{-3 \theta/2 - 1/2} + C \abs{m} R^{- \theta/2 - 1/2}.
    \end{align*}

	Since, as we noted, $|m| \leq \smallnorm{\omega(u)}_{L^1(\R^2)}$, the estimate above gives \refE{ugradTruL2Bound}.
	The bound on $\smallnorm{\grad \T_R u}_{L^2(\Omega_R)}$ is obtained in the same way except that 
	$\smallnorm{\grad u}_{L^2(\Sigma_R)}$ is replaced by 
	$\smallnorm{\grad u}_{L^2(\Omega_R \setminus \Sigma_R)} 
	\le \smallnorm{\grad u}_{L^2(\R^2)} = \smallnorm{\omega(u)}_{L^2(\R^2)}$. 
	But $\smallnorm{\omega}_{L^1(\R^2)} \le C \smallnorm{\omega}_{L^2(\R^2)}$ 
	because the support of $\omega$ is contained in $B_{R_0}$, 
	giving \refE{TRuL2Bound}.
\end{proof}

\begin{remark}

Had we not added the restriction that $\Omega$ is a disk, then since $\psi_\sigma$ is defined only up to an additive constant, we could choose this constant so that $\int_{\Sigma_R} \psi_\sigma = 0$. By normalizing $\psi_\sigma$ this way, though, the factor of $R^\theta$ in Poincare's inequality becomes the diameter $CR$ of $\Omega_R$ and we have 
\begin{align*}
	\smallnorm{\psi_\sigma}_{L^2(\Sigma_R)}
	   &\le C R \smallnorm{\sigma}_{L^2(\Sigma_R)}
	    \le C m R^{\theta/2 + 1/2}.
\end{align*}
This gives
\begin{align}\label{e:psiuL2SigmaRBoundNotADisk}
	&\smallnorm{\psi}_{L^2(\Sigma_R)}
		\le C C_0 R^{\theta/2 - 1/2} + C m R^{\theta/2 + 1/2}
\end{align}
and we would have had the following bound in the proof of \refP{Truncate2D}:
 	\begin{align*}
         &\smallnorm{u - \T_R u}_{L^2(\Omega_R)} \\
        	&\qquad
				\le C C_0 R^{\theta/2 - 3/2} + C m R^{\theta/2 - 1/2}
       				+ C C_0 R^{-\theta/2 - 1/2} + C m R^{-\theta/2 + 1/2} \\
			&\qquad\le 
				C C_0 R^{\theta/2 - 1/2} + C m R^{-\theta/2 + 1/2}.
    \end{align*}
    No value of $\theta$ in $[0, 1]$ will allow this to vanish as $R \to \iny$ when $m \ne 0$.
\end{remark}



%
%
\section{Estimates in 2D}\label{S:2DEstimates}

\noindent \begin{proof}[\textbf{Proof of \refP{FLILikeBasicEstimates} in 2D}]

\noindent Define $u^R$, $\varphi^R$, and $\Sigma_R$ as in \refS{uR}. Then
\begin{align}\label{e:uR2D}
	\begin{split}
		u^R 
		   &= \grad^\perp(\varphi^R \psi) 
		    = \varphi^R u + \psi \grad^\perp \varphi^R
	\end{split}
\end{align}
where $\psi = \psi_v + \psi_\sigma$ as in \refS{Truncation2D}.

The $L^1$ and $L^2$ norms (indeed, all $L^p$ norms) of the vorticity for solutions to ($E$) are conserved over time, while the bounds on the $L^2$ 
(for $m=0$) and $L^\iny$ norms of the velocity are bounded over any finite time interval (the former is conserved, and a bound on the latter is uniform). Thus, the estimates in \refP{Truncate2D} when applied to our solution $u$ to ($E$) are uniform in time. Hence, in the estimates that follow we will generally not explicitly refer to the bounds in time.

\skipline\noindent\textbf{Inequality 1}: follows directly from \refP{Truncate2D}.

\skipline\noindent\textbf{Inequality 2}: We have
\begin{align*}
    \smallnorm{u^R}_{L^\iny(\Omega_R)}
       &\le \smallnorm{\varphi^R u}_{L^\iny(\Omega_R)}
                + \smallnorm{\psi \grad^\perp \varphi^R}_{L^\iny(\Omega_R)} \\
       &\le \norm{u}_{L^\iny(\R^2)} + C R^{-\theta} \smallnorm{\psi}_{L^\iny(\Sigma_R)}.
\end{align*}
From \refE{psivBound}, $\norm{\psi_v}_{L^\iny(\Sigma_R)} \le C C_0/R$ and applying Poincare's inequality in the $L^\iny$ norm,
\begin{align*}
	\smallnorm{\psi_\sigma}_{L^\iny(\Sigma_R)}
		&\le C R^\theta \norm{\sigma}_{L^\iny(\Sigma_R)}
		\le C \abs{m} R^\theta R^{-1}
		\le C C_0 R^{\theta - 1}.
\end{align*}
We conclude that inequality (2) holds for sufficiently large $R$.

\skipline\noindent\textbf{Inequality 3}: We have,
\begin{align*}
    \smallnorm{\grad u^R}_{L^\iny}
       &\le \smallnorm{\grad u}_{L^\iny(\Omega_R)}
            + 2 \norm{\grad \varphi^R}_{L^\iny(\Sigma_R)}
                \smallnorm{u}_{L^\iny(\Sigma_R)} \\
       &\qquad\qquad
            + \smallnorm{\grad \grad^\perp \varphi^R}_{L^\iny(\Sigma_R)}
                \norm{\psi}_{L^\iny(\Sigma_R)} \\
       &\le C + C R^{-\theta} + C R^{-2 \theta} C_0 R^{\theta - 1}.
\end{align*}
We conclude that inequality (3) holds for sufficiently large $R$.

\skipline\noindent\textbf{Inequality 4a}: 
We begin with the observation that 
\begin{equation} \label{2Dparttpsi}
\Delta \psi_t = - \omega( \mbox{div } (u \otimes u) )
\end{equation}
and
\begin{equation} \label{2Dpressure}
-\Delta p = \mbox{ div } \mbox{div } (u\otimes u).
\end{equation}
Now, $u$ is bounded, uniformly over a finite time interval, in $L^4(\real^2)$. To see this write $u = v+\sigma$; clearly $\sigma \in L^4(\real^2)$ and 
$v \in L^{\infty} \cap L^2(\real^2)$. It follows that the right-hand side of both \eqref{2Dparttpsi} and \eqref{2Dpressure} are  second derivatives of an $L^2$ vector field. We can express $\psi_t$ and $p$ as  linear combinations of Riesz transforms of terms which are uniformly bounded in 
$L^2(\real^2)$ and hence, by the Calderon-Zygmund inequality, it follows that $\psi_t$ and $p$ are both  bounded, uniformly over a finite time interval, in 
$L^2(\real^2)$. 
 Thus,
\begin{align*}
   &\smallnorm{p \grad \varphi^R}_{L^2(\Omega_R)}
            + \smallnorm{\prt_t \psi \grad \varphi^R}_{L^2(\Omega_R)} \\
       &\qquad\le
              \smallnorm{p}_{L^2(\Sigma_R)}
              \smallnorm{\grad \varphi^R}_{L^\iny(\Sigma_R)}
            + \smallnorm{\prt_t \psi}_{L^2(\Sigma_R)}
              \smallnorm{ \grad \varphi^R}_{L^\iny(\Sigma_R)} \\
       &\qquad\le
            C R^{-\theta} + C R^{-\theta}
            = C R^{-\theta}.
\end{align*}

\skipline\noindent\textbf{Inequality 5}: For $f$ a scalar and $v$ a vector field we have that  
\[\Delta (f v) = \Delta f v + 2 \grad f \cdot (\grad v)^T + f \Delta v,\] so
\begin{align*}
	\Delta u^R 
	   &= \Delta(\varphi^R u) + \Delta(\psi \grad^\perp \varphi^R) \\
	   &= \Delta \varphi^R u + 2 \grad \varphi^R \cdot (\grad u)^T
	   			+ \varphi^R \Delta u
				+ \Delta \psi \grad^\perp \varphi^R \\
			&\qquad\qquad
				+ 2 \grad \psi \cdot (\grad \grad^\perp \varphi^R)^T
				+ \psi \Delta \grad^\perp \varphi^R.
\end{align*}
Then because $\Delta \psi = \omega$ and $u = \grad^\perp \psi$,
\begin{align*}
	\norm{\Delta u^R}_{L^2(\Omega_R)}
		&\le C R^{-2 \theta} \norm{u}_{L^2(\Sigma_R)}
			+ C R^{-\theta} \norm{\grad u}_{L^2(\Sigma_R)} \\
		&\qquad
			+ C \norm{\Delta u}_{L^2(\Omega_R)} 
			+ C R^{-\theta} \norm{\omega}_{L^2(\Sigma_R)} \\
		&\qquad
			+ C R^{-2 \theta} \norm{u}_{L^2(\Sigma_R)}
			+ C R^{-3 \theta} \norm{\psi}_{L^2(\Sigma_R)} \\
		&\le C \norm{\Delta u}_{L^2(\R^2)} + C,
\end{align*}
for sufficiently large $R$ by the bounds in \refE{psiuL2SigmaRBound}, \refE{uL2SigmaRBound}, and \refE{graduL2SigmaRBound}.

\skipline\noindent\textbf{Inequality 6}: We have,
$\smallnorm{u^R - u}_{L^2(\Omega_R)} \le C C_0 R^{-\al}$ by \refE{uTruL2Bound} and
\begin{align*}
	&\smallnorm{u^R - \varphi^R u}_{L^2(\Omega_R)}
			\le \smallnorm{u^R - u}_{L^2(\Omega_R)}
				+ \smallnorm{(1 - \varphi_R) u}_{L^2(\Omega_R)} \\
		&\qquad
			\le \smallnorm{u^R - u}_{L^2(\Omega_R)}
				+ \smallnorm{u}_{L^2(\Sigma_R)} \\
		&\qquad
			\le C C_0 R^{-\al} + C C_0 R^{\theta/2 - 3/2} 
			+ C \abs{m} R^{\theta/2 - 1/2} \\
		&\qquad
			\le C C_0 R^{-\al}
\end{align*}
where we used \refE{uL2SigmaRBound} and \refE{uTruL2Bound}.

\skipline\noindent\textbf{Inequality 7}: the two-dimensional case follows directly from \refP{Truncate2D}.
\end{proof}

%
%
\section{Decay of velocity in 3D}\label{S:Decay3D}

\noindent The basic existence result for a solution to ($E$) is given in \refT{EomegaBounds}. 
\begin{theorem}\label{T:EomegaBounds}
Assume that $u_0$ is in $C^s(\R^3) \cap V_C(\R^3)$ for $s > 1$. Then there exists $T^*$ in $(0, \iny]$ such that for all $T$ in $(0, T^*)$ the solution $u$ to ($E$) in the whole space lies in $L^\iny([0, T]; C^s(\R^3))$ with $\grad p$ in $L^\iny([0, T]; L^a(\R^3))$ for all $a$ in $(1, \iny]$. Also, the vorticity $\omega = \omega(u)$ lies in $L^\iny([0, T]; L^a(\R^3))$ for all $a$ in $[1, \iny]$ with a bound on its norm that is independent of $a$. Furthermore, $\omega$ remains compactly supported for all time, with the support contained in a ball of radius $R(T) \le R_0 + \norm{u}_{L^\iny([0, T] \times \R^3)} T$ for all $t$ in $[0, T]$, where the support of $\omega_0$ is contained in $B_{R_0}(0)$.
\end{theorem}
\begin{proof}
The initial vorticity $\omega = \omega(u_0)$ is compactly supported and continuous by assumption so $\grad v_0$ is in $L^a(\R^3)$ for all $a$ in $[1, \iny]$. Thus by Theorem 4.2.3 p. 79 of \cite{C1998},  there exists $T^*$ in $(0, \iny]$ such that for all $T$ in $(0, T^*)$ the solution $u$ to ($E$) lies in $L^\iny([0, T]; C^s(\R^3))$ with $\grad p$ in $L^\iny([0, T]; L^a(\R^3))$ for all $a$ in $(1, \iny)$.

Taking the vorticity of ($E$) gives
\begin{align*}
	\frac{D \omega}{D t}
		&= \prt_t \omega + u \cdot \grad \omega = \omega \cdot \grad u,
\end{align*}
where $D \omega/D t$ is the rate of change of the vorticity along the flow lines. Assume that the support of $\omega^0$ is contained in a ball $B_{R(0)}$ of radius $R(0)$. Then until some flow line starting from within $B_{R(0)}$ reaches a point outside of $B_{R(0)}$ there can be no change in vorticity at that point. This gives the compact support of the vorticity for all time along with the bound on its support.
\end{proof}

\begin{remark}\label{R:EomegaBounds2D}
	\refT{EomegaBounds} is stated for three dimensions but holds in two dimensions as 
	well, with $T^* = \iny$, with minor adaptations to allow for $u_0 \in C^s (\real^2) \cap E_m$, $m \neq 0$.
\end{remark}

\begin{cor}\label{C:uLpDecay}

Let $u$ be a solution to ($E$) as in \refT{EomegaBounds}. Then $u$ is in $L^\iny([0, T]; L^p(\R^3))$ for all $p$ in $(3/2, \iny]$ and $\grad u$ is in $L^\iny([0, T]; L^p(\R^3))$ for all $p$ in $(1, \iny]$. Also,
\begin{align*}
	\abs{u(x)} \le C/\abs{x}^2, \quad
	\abs{\grad u(x)} \le C/\abs{x}^3
\end{align*}
for all $t$ in $[0, T]$ and $\abs{x} \ge 2 R(T)$, and
\begin{align*}
	\norm{u(t)}_{L^2(\Sigma_R)}
		\le C R^{-1/2}, \quad
	\norm{\grad u(t)}_{L^2(\Sigma_R)}
		\le C R^{-3/2}, \quad
\end{align*}
for all $t$ in $[0, T]$ and all $R$ sufficiently large that $\Sigma_R$ is contained in $B_{2R(T)}^C$.
\end{cor}
\begin{proof}
Let $\Psi$ be the associated stream function in three dimensions and recall the expression for $\Psi$ given in \eqref{3Dstream}. In particular, 
 $u = \mbox{curl } \Psi$ and $-\Delta \Psi = \omega$, where $\omega = \omega(u)$ is the associated vorticity. This  gives rise to the three-dimensional Biot-Savart law, allowing us to write the velocity in terms of vorticity as: 
\begin{equation} \label{3DBSlaw}
u=u(x) = \frac{-1}{4\pi}\int \frac{x-y}{|x-y|^3} \times \omega (y)\,dy.
\end{equation} 
The decay of $u$ and $\grad u$ then follow from the compact support of $\omega$ and the Biot-Savart law. Because $u$ and $\grad u$ are in 
$L^\iny([0, T] \times \R^3)$, the membership of $u$ and $\grad u$ in the stated spaces and the decay of their $L^2$ norms then follow from their decay at infinity.
\end{proof}

%
%
\section{Truncation operator in 3D}\label{S:Truncation3D}

\noindent Let $\varphi^R$ and $\Sigma_R$ be as in \refS{uR} and recall the definition of the truncation operator in three dimensions: if  
$u$ in $H^1(\R^3)$ then $\T_R u$ is given by \eqref{e:TRu3D}; that is, $\T_R u = \grad \times (\varphi^R \Psi)$. Set $E=E(x) = (4\pi|x|)^{-1}$, the fundamental solution of $-\Delta$ in $\real^3$. 

\begin{lemma}\label{L:EfDecay}
If $f$ in $L^2(\R^3) \cap L^\iny(\R^3)$ is compactly supported in the ball $B_L$ of radius $L$ centered at the origin then for all $R \ge R^*$,
\begin{align*}
	\norm{E * f}_{L^2(\Sigma_R)}
		&\le C L^{3/2} \norm{f}_{L^2(\R^3)} R^{1/2}, \\
	\norm{\prt_k E * f}_{L^2(\Sigma_R)}
		&\le C L^{3/2} \norm{f}_{L^2(\R^3)} R^{-1/2}, \\	
	\norm{E * f}_{L^\iny(\Sigma_R)}
		&\le C L^{3/2} \norm{f}_{L^2(\R^3)} R^{-1}, \\
	\norm{\prt_k E * f}_{L^\iny(\Sigma_R)}
		&\le C L^{3/2} \norm{f}_{L^2(\R^3)} R^{-2},
\end{align*}
where $C$ is an absolute constant and $R^*$ depends only upon $\Omega_1$ and $L$.
\end{lemma}
\begin{proof}
On $B_{2L}^C$ the compact support of $f$ gives $\abs{E*f} \le C \norm{f}_{L^1(\R^3)}/\abs{x} \le C L^{3/2} \norm{f}_{L^2(\R^3)}/\abs{x}$. Therefore, as long as $R$ is large enough so that $\Sigma_R$ is wholly contained in $B_{2L}^C$,
\begin{align*}
	\norm{E * f}_{L^2(B_R)}
		&\le C L^{3/2} \norm{f}_{L^2(\R^3)} R^{1/2},
\end{align*}
and similarly for the $L^\iny$ norm and for the inequalities for $\prt_k E * f$.
\end{proof}
 
\begin{prop}\label{P:Truncate3D}
Let $u$ be a solution to ($E$) as in \refT{EomegaBounds}. Then there exists $R^* > 0$ such that for all $R \ge R^*$ and all $t$ in $[0, T]$, $\T_R u$ is in $V(\Omega_R)$, and
    \begin{align}\label{e:T3DVCE}
        \smallnorm{u - \T_R u}_{L^2(\Omega_R)} \le C_1 R^{-1/2},
    \end{align}
    \begin{align}\label{e:TRomegaDiffBound}
        \smallnorm{\grad(u - \T_R u)}_{L^2(\Omega_R)} \le C_2 R^{-3/2}.
    \end{align}
    The constants $C_1$ and $C_2$ depend on $u$ and $T$, and $R^*$ depends
    on $R(T)$ and $\Omega_1$.
\end{prop}
\begin{proof}
On $\Gamma_R$, $\T_R u$ vanishes identically, while clearly $\mbox{div }\T_R u=0$   on $\Omega_R$. 

Now, 
$\T_R u = \varphi_R u + \nabla \varphi_R \times \Psi$, so 
\begin{align*} 
	    &\smallnorm{u - \T_R u}_{L^2(\Omega_R)} \\
	        &\qquad
	        	\le \smallnorm{u}_{L^2(\Sigma_R)}
	                + C\norm{\grad \varphi^R}_{L^\iny(\Sigma_R)}
	                    \norm{\Psi}_{L^2(\Sigma_R)} \\
	        &\qquad
	        	\le \smallnorm{u}_{L^2(\Sigma_R)}
	                + C R^{-1} \sum_{i, k}
	                        \smallnorm{E * \omega (u)}_{L^2(\Sigma_R)}
	            \le C R^{-1/2}.
\end{align*}
We used \refC{uLpDecay} and \refL{EfDecay} in the final inequality.

Also,
\begin{equation} \label{nablauR}
\partial_k\T_R u = \varphi_R \partial_k u + \nabla \partial_k\varphi_R \times \Psi + (\partial_k \varphi_R) u +  \nabla \varphi_R \times \partial_k \Psi.
\end{equation}
Hence,  
\begin{align*} 
		&\smallnorm{\grad \T_R u - \grad u}_{L^2(\Omega_R)} \\
			&\qquad \le \smallnorm{\grad u}_{L^2(\Sigma_R)}
				+ \norm{\grad \grad \varphi^R}_{L^\iny(\Sigma_R)}
					\norm{\Psi}_{L^2(\Sigma_R)} \\
			&\qquad\qquad
				+ \norm{\grad \varphi^R}_{L^\iny(\Sigma_R)}
				\smallnorm{u}_{L^2(\Sigma_R)} \\
			&\qquad\qquad
				+ \norm{\grad \varphi^R}_{L^\iny(\Sigma_R)} \sum_{k}
				\norm{\prt_k E * \omega (u)}_{L^2(\Sigma_R)} \\
		&\qquad \le \smallnorm{\grad u}_{L^2(\Sigma_R)} + C R^{-2}
				 \norm{\Psi}_{L^2(\Sigma_R)} \\
			&\qquad\qquad + C R^{-1} \smallnorm{u}_{L^2(\Sigma_R)}
				+ C R^{-1} \sum_{k}
		 		\norm{\prt_k E * \omega (u)}_{L^2(\Sigma_R)} \\
		&\qquad \le
			C R^{-3/2},
\end{align*}
where we used \refC{uLpDecay} and \refL{EfDecay} in the final inequality. Thus, $\T_R u \in V(\Omega_R)$, as desired.
\end{proof}

%
%
\section{Estimates in 3D}\label{S:3DEstimates}

\noindent \begin{proof}[\textbf{Proof of \refP{FLILikeBasicEstimates} in 3D}]

\noindent Define  $\varphi^R$ and $\Sigma_R$ as in \refS{uR}. Then
\begin{equation} \label{e:uR3D}
 		u^R = \T_R u = \varphi_R u + \nabla \varphi_R \times \Psi.  
\end{equation}

\skipline\noindent\textbf{Inequality 1}: follows from \refE{TRomegaDiffBound}.

\skipline\noindent\textbf{Inequality 2}: From \refE{uR3D},
\begin{align*}
    \smallnorm{u^R}_{L^\iny(\Omega_R)}
       &\le \smallnorm{\varphi^R u}_{L^\iny(\Omega_R)}
                + \smallnorm{\grad \varphi^R}_{L^\iny(\Omega_R)}
                   \smallnorm{\Psi}_{L^\iny(\Sigma_R)} \\
       &\le \norm{u}_{L^\iny(\R^2)} + C R^{-1}   
       		\smallnorm{E * \omega}_{L^\iny(\Sigma_R)},
\end{align*}
which, using \refL{EfDecay}, is enough to establish inequality (2).

\skipline\noindent\textbf{Inequality 3}:  Recall the expression for $\nabla u^R$ in \eqref{nablauR}.
Then, using \refL{EfDecay}, we have
\begin{align*}
    \smallnorm{\grad u^R}_{L^\iny}
       &\le \smallnorm{\grad u}_{L^\iny(\Omega^R)}
            + \smallnorm{\grad \grad \varphi^R}_{L^\iny(\Sigma^R)}
                 \norm{\Psi}_{L^\iny(\Sigma^R)} \\
        &\qquad\qquad
            + \norm{\grad \varphi^R}_{L^\iny(\Sigma^R)}
                \smallnorm{u}_{L^\iny(\Sigma^R)} \\
       &\qquad\qquad
             + \smallnorm{\grad \varphi^R}_{L^\iny(\Sigma^R)}
                \sum_{k} \norm{\prt_k E * \omega}_{L^\iny(\Sigma^R)} \\           
       &\le C + C R^{-3} 
       			+ C R^{-1} + C R^{-3}
		\le C.
\end{align*}

\skipline\noindent\textbf{Inequality 4b}: The proof of this inequality is similar to that of inequality 4a. We begin by observing that $\Psi_t$ satisfies an equation analogous to \eqref{2Dparttpsi}, namely
\[\Delta\Psi_t = \mbox{curl }\mbox{div } u\otimes u.\]
Now, $u \in L^2(\real^3) \cap L^{\infty}(\real^3)$, uniformly in time up to $T^*$. Therefore, using again the Calderon-Zygmund inequality, it follows 
that $\Psi_t \in L^2(\real^3)$, uniformly in time up to $T^*$. The estimate for the pressure $p$ in three dimensions is exactly the same as for two dimensions, using \eqref{2Dpressure}. Thus, using the scaling of $\varphi_R$, we get 
\begin{align*}
   &\smallnorm{p \grad \varphi^R}_{L^2(\Omega_R)}
            + \smallnorm{\grad \varphi^R \cdot \prt_t \psi}_{L^2(\Omega_R)} \\
       &\qquad\le
              \smallnorm{p}_{L^2(\Sigma_R)} 
              \smallnorm{\grad \varphi^R}_{L^\iny(\Sigma_R)}
            + \smallnorm{\prt_t \psi}_{L^2(\Sigma_R)}
              \smallnorm{ \grad \varphi^R}_{L^\iny(\Sigma_R)} \\
       &\qquad\le
            C R^{-1} + C R^{-1}
            = C R^{-1}.
\end{align*}

\skipline\noindent\textbf{Inequality 5}: We must compute $\Delta u^R$; for the sake of simplicity we keep track mostly of the order of derivatives. We have then that
\begin{align*}
	\Delta u^R
 		&= (\Delta \varphi^R) u + 2 \grad \varphi^R \cdot D u
	   			+ \varphi^R \Delta u + \nabla  \varphi^R \times \omega   \\
			&\qquad\qquad
				+ 2 D^2 \varphi^R \cdot D \Psi  
				+ \nabla \Delta \varphi^R \times \Psi .
\end{align*}
Using \refC{uLpDecay} and \refL{EfDecay}, we find
\begin{align*}
	\norm{\Delta u^R}_{L^2(\Omega_R)}
		&\le C R^{-2} \norm{u}_{L^2(\Sigma_R)}
			+ C R^{-1} \norm{\grad u}_{L^2(\Sigma_R)} \\
		&\qquad
			+ C \norm{\Delta u}_{L^2(\Omega_R)} 
			+ C R^{-1} \norm{\omega}_{L^2(\Sigma_R)} \\
		&\qquad
			+ C R^{-2}\norm{D\Psi}_{L^2(\Sigma_R)} 
			+ C R^{-3}\norm{\Psi}_{L^2(\Sigma_R)}\\
		&\qquad \le C + CR^{-5/2} \le C.
\end{align*}

\skipline\noindent\textbf{Inequality 6}: We have,
$\smallnorm{u^R - u}_{L^2(\Omega_R)} \le C R^{-1/2}$ by \refE{T3DVCE} and
\begin{align*}
	&\smallnorm{u^R - \varphi^R u}_{L^2(\Omega_R)}
			\le \smallnorm{u^R - u}_{L^2(\Omega_R)}
				+ \smallnorm{(1 - \varphi_R) u}_{L^2(\Omega_R)} \\
		&\qquad
			\le \smallnorm{u^R - u}_{L^2(\Omega_R)}
				+ \smallnorm{u}_{L^2(\Sigma_R)}
			\le C R^{-1/2},
\end{align*}
where we used \refC{uLpDecay}.

\skipline\noindent\textbf{Inequality 7}: the three-dimensional case follows directly from \refP{Truncate3D}.
\end{proof}

%
%
\section{Truncation of the initial velocity}\label{S:u0Truncation}

\noindent The following observations hold in any dimension.

Define $\W_R \colon V(\R^d) \to H(\Omega_R)$ as follows: for any $u$ in
$V(\R^d)$ let $\W_R u$ be that unique vector in $H(\Omega_R)$ having the same
vorticity as $u$ ($\Omega$ is simply connected so the vector is unique), see the paragraph following the statement of \refT{MainResult}.
Let $P_{H(\Omega_R)}: V(\R^2) \to H(\Omega_R)$ be the projection operator into
$H(\Omega_R)$.
\begin{lemma}\label{L:WEqualsPH}
    The operators $\W_R$ and $P_{H(\Omega_R)}$ are the same.
\end{lemma}
\begin{proof}
    Let $u$ be in $V(\R^d)$. Then by the Leray-Helmholtz decomposition there exists
    $v$ in $H(\Omega_R) \cap H^1(\Omega_R)$ and $p$ in $H^2(\Omega_R)$ such that $u
    = v + \grad p$ on $\Omega_R$, with $\Delta p = 0$ and with $v$ and $\grad p$
    unique. Then $v$ is orthogonal to $\grad p$ in $L^2(\Omega_R)$ (that is, in the
    $H(\Omega_R)$-norm) so $v = P_{H(\Omega_R)} u$. But the curl of a gradient is
    zero, so $\omega(v) = \omega(u)$ which shows that $v = \W_r u$ as well.
\end{proof}
\Ignore{  
    We prove this first in two dimensions. Let $u$ be in $V(\R^2)$ and let $v = P_{H(\Omega_R)}
    u$. By definition, $v$ is that vector in $H(\Omega_R)$ that minimizes $\norm{u
    - v}_{L^2(\Omega_R)}$. Thus, fixing any $w$ in $H(\Omega_R)$ and letting
    \begin{align*}
        f(\eps) = \norm{u - (v + \eps w)}_{L^2(\Omega_R)}^2
    \end{align*}
    we must have
    \begin{align*}
        0
        &= f'(0)
            = \int_{\Omega_R} (u - v) \cdot w.
    \end{align*}

    Let $\psi$ be the stream function for $w$. Since $w \cdot \mathbf{n} = 0$ on
    $\Gamma$ we can assume that $\psi = 0$ on $\Gamma$. Then
    \begin{align*}
        \int_{\Omega_R} (u - v) \cdot w
        &= \int_{\Omega_R} (u - v) \grad^\perp \psi
            = \int_{\Omega_R} (u - v)^\perp \grad \psi \\
        &= - \int_{\Omega_R} \dv (u - v)^\perp \psi
            = \int_{\Omega_R} (\omega(u) - \omega(v)) \psi.
    \end{align*}
    Since this must be true for all $\psi$ in $H_0^1(\Omega)$ it follows that
    $\omega(v) = \omega(u)$.
    }

\refC{WRProperties} follows immediately from \refL{WEqualsPH}, given the well-known properties of the projection operator.

\begin{cor}\label{C:WRProperties}
    $\W_R$ has the same convergence properties in the $L^2(\Omega_R)$-norm of the
    velocity as does $\T_R$. In particular, we obtain \eqref{e:u0WDecay}.
\end{cor}

\begin{remark}
    The operator $\W_R$ is a more natural ``truncation'' operator than $\T_R$ since
    it involves literally truncating the vorticity. It is perfectly adequate for
    truncating the initial velocity for the Navier-Stokes equations, but because we
    do not know how to control the remainder following truncation in the
    $H^1$-norm, we cannot use it to construct the approximate solution $u^R$ to the
    Euler equations.
\end{remark}

%
%
\Ignore{
    \begin{remark}

    It is shown in Section 4 of \cite{K2005UBD} that without the assumption of
    compact support of the initial vorticity, for all $u$ in $V(\R^2)$
    (corresponding to class I of \refD{Initu0}),
    \begin{align*}
        \smallnorm{u - \T_R u}_{L^2(\Omega_R)} \to 0 \text{ and }
        \smallnorm{\grad(u - \T_R u)}_{L^2(\Omega_R)} \to 0
    \end{align*}
    as $R \to \iny$.
    \end{remark}
    }

\section{Comments and conclusions} \label{S:Comments}

As a first comment, we note that the estimates in Sections~\ref{S:Decay3D} and \ref{S:Truncation3D} for three dimensions are considerably simplified by the assumption that the initial vorticity is compactly supported, whereas for two dimensions this assumption is merely a minor convenience that allows us in \refT{MainResult} to give an explicit rate of convergence in $R$. If one drops the assumption of compact support, it does not seem possible to obtain a uniform-in-time bound on the decay of the gradient of the velocity. One can obtain a bound on the decay of vorticity, however, if one assumes that the $L^2$ norm of the initial vorticity on a ball of radius $R$ decays for large $R$ at least as fast as $C R^{-1/2}$. One then modifies the energy argument in \refS{uProofOfMainResult}, integrating by parts differently to use the decay of the vorticity in place of the decay of the velocity. The value of $\al$ in \refT{MainResult}, however, must be strictly less than $1/2$.

The case of initial vorticity with nonzero integral in two dimensions corresponds to the situation where 
the limiting full plane flow has infinite energy. Since our argument is based on energy estimates, it is 
natural that this situation would be complicated. In this situation we have only studied the case of the 
expanding disk, but, as explained in Section 5, this is not a matter of convenience. Our argument makes 
essential use of the fact that we are working in a disk. However, the restriction to the expanding disk is by
no means natural, and removing this assumption becomes an interesting open problem.

Finally, we note that our result was proved for smooth flows but that, in two dimensions, it would be very reasonable to consider initial vorticities in $L^p$, $p>2$, compactly supported, given that, in our proof, the high regularity was needed in estimates near the boundary, where vorticity acts as far field. If $p>2$ then the support of vorticity can be controlled since the Euler velocity is {\it a priori} bounded.

%
%
\section*{Acknowledgements}

\noindent  The research of the first author was supported in part by NSF grant DMS-0705586, the second author was supported by CNPq grant \# 303301/2007-4, and the third author was supported by CNPq grant \# 302214/2004-6. This work gratefully ackowledges the support of the FAPESP Thematic Project \# 2007/51490-7. The first author wishes to thank the hospitality of IMECC-UNICAMP, where a large portion of this work was done, along with FAPESP grant \# 2007/03713-7, for supporting his visit.

The authors are grateful for helpful discussions with Franck Sueur.


\begin{thebibliography}{1}

\bibitem{C1998}
Jean-Yves Chemin.
\newblock {\em Perfect incompressible fluids}, volume~14 of {\em Oxford Lecture
  Series in Mathematics and its Applications}.
\newblock The Clarendon Press Oxford University Press, New York, 1998.
\newblock Translated from the 1995 French original by Isabelle Gallagher and
  Dragos Iftimie.

\bibitem{DM1987}
Ronald~J. DiPerna and Andrew~J. Majda.
\newblock Concentrations in regularizations for {$2$}-{D} incompressible flow.
\newblock {\em Comm. Pure Appl. Math.}, 40(3):301--345, 1987.

\bibitem{FLI2003} 
Drago\c{s}~Iftimie, Milton~C. Lopes~Filho, and Helena~J. Nussenzveig~Lopes.
\newblock Two dimensional incompressible ideal flow around a small obstacle.
\newblock {\em Comm. Part. Diff. Eqns.}, 28(1-2):349--379, 2003.

\bibitem{FLI2006}
Drago\c{s}~Iftimie, Milton~C. Lopes~Filho, and Helena~J. Nussenzveig~Lopes.
\newblock Two-dimensional incompressible viscous flow around a small obstacle.
\newblock {\em Math Ann},  336(2):449--489, 2006.

\bibitem{FLI2007}
Drago\c{s}~Iftimie, Milton~C. Lopes~Filho, and Helena~J. Nussenzveig~Lopes.
\newblock Incompressible flow around a small obstacle and the vanishing viscosity limit.
\newblock {\em To appear in Commun. in Math. Phys.} Cambridge University Press, Cambridge, 2008. 
Available online, http://dx.doi.org/10.1007/s00220-008-0621-3.

\bibitem{IK2008}
Drago\c{s}~Iftimie, James~P. Kelliher.  
\newblock Remarks on the vanishing obstacle limit for a 3D viscous incompressible fluid.
\newblock {\em To appear in Proc. Amer. Math. Soc.}.

\bibitem{Kato1983}
Tosio Kato.
\newblock Remarks on zero viscosity limit for nonstationary {N}avier-{S}tokes
  flows with boundary.
\newblock In {\em Seminar on nonlinear partial differential equations
  (Berkeley, Calif., 1983)}, volume~2 of {\em Math. Sci. Res. Inst. Publ.},
  pages 85--98. Springer, New York, 1984.

\bibitem{K2005UBD}
James~P. Kelliher.
\newblock Expanding domain limit for incompressible fluids in the plane.
\newblock {\em Comm. Math. Phys.}, 278(3):753--773, 2008.

\bibitem{KStatSol2008}
James~P. Kelliher.
\newblock Infinite-energy 2D statistical solutions to the equations of incompressible fluids.
\newblock {\em Preprint}.

\bibitem{Lacave2008a}
Christophe Lacave.
\newblock Two Dimensional Incompressible Ideal Flow Around a Thin Obstacle Tending to a Curve.
\newblock {\em To appear in Ann. IHP: Anal. non-Lin.}.

\bibitem{Lacave2008b}
Christophe Lacave.
\newblock Two Dimensional Incompressible Viscous Flow Around a Thin Obstacle Tending to a Curve.
\newblock {\em Submitted}.

\bibitem{Lopes2006}
Milton~C. Lopes~Filho.
\newblock Vortex dynamics in a two-dimensional domain with holes and the small obstacle limit.
\newblock {\em SIAM J. Math. Anal.}, 39(2):422--436, 2007.

\bibitem{MB2001}
Andrew~J. Majda and Andrea~Bertozzi.
\newblock {\em Vorticity and incompressible flow}.
\newblock Cambridge University Press, Cambridge, UK, 2002.
\newblock Cambridge Texts in Applied Mathematics, 27. 
 
\bibitem{T2001}
Roger Temam.
\newblock {\em Navier-{S}tokes equations}.
\newblock AMS Chelsea Publishing, Providence, RI, 2001.
\newblock Theory and numerical analysis, Reprint of the 1984 edition.

\end{thebibliography}

\end{document}